\documentclass[preprint,11pt]{elsarticle}     

\usepackage{hyperref}

\usepackage{float}
\usepackage{amstext,amsfonts,amsmath,amssymb,array,amsthm}

\usepackage{graphicx,xcolor}
\makeatletter
\def\@author#1{\g@addto@macro\elsauthors{\normalsize%
		\def\baselinestretch{1}%
		\upshape\authorsep#1\unskip\textsuperscript{%
			\ifx\@fnmark\@empty\else\unskip\sep\@fnmark\let\sep=,\fi
			\ifx\@corref\@empty\else\unskip\sep\@corref\let\sep=,\fi}%
		\def\authorsep{\unskip,\space}%
		\global\let\@fnmark\@empty
		\global\let\@corref\@empty  
		\global\let\sep\@empty}%
	\@eadauthor={#1}
}
\makeatother

\newtheorem{theorem}{Theorem}[section]

\newtheorem{corollary}[theorem]{Corollary}
\newtheorem{lemma}[theorem]{Lemma}

\theoremstyle{definition}
\newtheorem{example}[theorem]{Example}

\def\simp{\overset{\scriptscriptstyle{+}}{\sim}}
\def\simm{\overset{\scriptscriptstyle{-}}{\sim}}

\journal{DM.}

\begin{document}

\begin{frontmatter}

\title{\bf On eigenvalue multiplicity in signed graphs}

\author{Farzaneh Ramezani}
\ead{f.ramezani@kntu.ac.ir}
\address{Department of Mathematics, K.N. Toosi University of Technology\\ P.O. Box 16315-1618, Teheran, Iran.}

\author{Peter Rowlinson}
\ead{p.rowlinson@stirling.ac.uk}
\address{Mathematics and Statistics Group,
	Division of Computing Science and Mathematics, 
	University of Stirling\\ Scotland FK9 4LA, United Kingdom.}

\author{Zoran Stani\' c}
\ead{zstanic@math.rs}
\address{Faculty of Mathematics, University of Belgrade \\
Studentski trg 16, 11 000 Belgrade, Serbia.}

\begin{abstract}Given a signed graph $\Sigma$ with $n$ vertices, let $\mu$ be an eigenvalue of $\Sigma$, and let~$t$ be the codimension of the corresponding eigenspace. We prove that
	$$n\leq{t+2\choose 3}$$
	whenever $\mu\notin\{0, 1, -1\}$. We show that this bound is sharp by providing examples of signed graphs in which it is attained.  We also discuss particular cases in which the bound can be decreased.
\end{abstract}


\end{frontmatter}

\section{Introduction}

A \textit{signed graph} $\Sigma$ is a pair $(G, \sigma)$, where $G=(V, E)$ is an (unsigned) graph, called the \textit{underlying graph}, and $\sigma\colon E\longrightarrow\{1, -1\}$ is the \textit{sign function} or the \textit{signature}.  The edge set $E$ consists  of positive and negative edges. Throughout the paper we interpret a  graph  as a signed graph in which all edges are positive.

The \textit{degree} $d_i$ of a vertex $i$ coincides with its degree in the underlying graph. Its \textit{net-degree} $d_i^{\pm}$ is the difference between the numbers of positive and negative edges incident with it.
A signed graph is called   \textit{net-regular} if the net-degree is constant on the vertex set,  and it is called {\em regular} if the underlying graph is regular.

The \textit{adjacency matrix} $A_{\Sigma}$ of $\Sigma$ is obtained from the adjacency matrix of its underlying graph by reversing the sign of all 1s
which correspond to negative edges.
The \textit{eigenvalues} of $\Sigma$ are identified as the eigenvalues of $A_{\Sigma}$.

Let $\mu$ be an eigenvalue of $\Sigma$ with multiplicity $k$, and let $t=n-k$. Then~$k$ is the dimension of the eigenspace $\mathcal{E}(\mu)$ and $t$ is its codimension, i.e., the dimension of $\mathcal{E}(\mu)^\perp\subseteq \mathbb{R}^n$. It is not difficult to show that if $\mu\in\{0, 1, -1\}$, then $n$ is not bounded above by a function of $t$.

The spectral theory of signed graphs, which nicely encapsulates the spectral theory of graphs, has received a significant deal of attention in the recent past.  In this study we use the concept of star complements in signed graphs to extend some results of \cite{BeRo} concerning graphs to the class of signed graphs. In particular, we prove that, for 
$\mu\notin\{0, 1, -1\}$,
\begin{equation}\label{eq:nleq}
n\leq {t+2\choose 3}.\end{equation}
We provide examples of signed graphs attaining the bound \eqref{eq:nleq} and consider certain cases in which the bound can be decreased; for example, we prove that the bound is decreased by 1 for any signed graph switching equivalent to a net-regular signed graph.

In Section~\ref{sec:2} we fix some terminology and notation and list some preliminary results. This section also contains a brief introduction to the concept of star complements in signed graphs. Our contribution is reported in Sections~\ref{sec:main} and~\ref{sec:pc}.

\section{Preliminaries}\label{sec:2}

If the vertices $u$ and $v$ of a signed graph $\Sigma$ are adjacent, we write $u\sim v$; in particular, the existence of a positive (resp.~negative) edge between these vertices is designated by $u\simp v$ (resp.~$u\simm v$). The set of vertices adjacent to a vertex $u$ is denoted by $N(u)$. A signed graph is said to be \textit{homogeneous} if all its edges have the same
sign (in particular, if its edge set is empty). Otherwise, it is \textit{inhomogeneous}.

The \textit{negation} $-\Sigma$, of $\Sigma$, is obtained by reversing the sign of all edges of~$\Sigma$.  Thus it has  $-A_{\Sigma}$ 
as an adjacency matrix.

For $U$ a subset of the vertex set $V(\Sigma)$, let $\Sigma^U$ be the signed graph
obtained from $\Sigma$ by reversing the sign of each edge between
a vertex in $U$ and a vertex in $V(\Sigma) \setminus U$. The
signed graph $\Sigma^U$ is said to be {\it switching equivalent}
to $\Sigma$. Switching equivalent signed graphs share the same spectrum.

As  noted in \cite{Zas}, the all-1 vector $\mathbf{j}_n$ is an eigenvector of a signed graph $\Sigma$ if and only if $\Sigma$ is net-regular, and then $\mathbf{j}_n$ belongs to the eigenspace of its net-degree.

We turn to star complements. Let $\mu$ be an eigenvalue of $\Sigma$ with multiplicity $k$, and let $P$ be the matrix representing the orthogonal
projection of $\mathbb{R}^n$ onto the eigenspace
$\mathcal{E}(\mu)$ with respect to the canonical basis $\{\mathbf{e}_1, \mathbf{e}_2, \ldots, \mathbf{e}_n\}$ of $\mathbb{R}^n$. A set
$S\subseteq V(\Sigma)=\{1, 2, \ldots, n\}$, such that the vectors
$P\mathbf{e}_u~(u\in S)$ form a basis of $\mathcal{E}(\mu)$,
is called a {\em star set} for $\mu$, while the signed
graph induced by the set $\smash{\overline{S}=V(\Sigma)\setminus S}$ is called a {\em star complement} for $\mu$.  We note that $|S|=k$ and  write  $t=n-k$. 

Many results on star complements in graphs can be found in \cite{CvRoSi-1,CvRoSi-2}, and some of them are transferred to the class  of signed graphs in recent publications, for example \cite{StaF}. In particular,  $\mu$ is not an eigenvalue of any star complement for $\mu$. Moreover 
if $\mu\neq 0$, then every vertex of $S$ is adjacent to
at least one vertex of $\overline{S}$, and if $\mu\notin\{0, 1, -1\}$, then any  two vertices of $S$ have distinct neighbourhoods in~$\overline{S}$ \cite{StaF}.  Therefore, when $\mu\notin\{0, 1, -1\}$, we have $k\leq 3^t-1$, i.e., $n\leq 3^t+t-1$. This is an initial upper bound for $n$, which will be improved to the cubic function of $t$ given in~\eqref{eq:nleq}.

The following theorem,  called the Reconstruction Theorem, is a straightforward analogue of 
\cite[Theorem 5.1.7]{CvRoSi-2}.

\begin{theorem}\label{theo:rec}
	Let $\Sigma$ be a signed graph with adjacency matrix
	\[
	\begin{pmatrix}A_S&B^T\\B&C\end{pmatrix},
	\]
	where $A_S$ is the $k\times k$ adjacency matrix of the  subgraph
	induced by a vertex set $S$. Then $S$ is a star set for
	$\mu$ if and only if $\mu$ is not an eigenvalue of $C$ and
	\begin{equation*}\label{rct}\mu I-A_S = B^T(\mu
	I-C)^{-1}B.\end{equation*}
\end{theorem}

If $S$ is a star set for $\mu$, then $C$ is the adjacency matrix of the corresponding star complement.  Now  for  $\mathbf{x}, \mathbf{y}\in \mathbb{R}^t$, we define the bilinear form
$$\langle \mathbf{x}, \mathbf{y}\rangle=\mathbf{x}^T(\mu I-C)^{-1}\mathbf{y}.$$
Here is a direct consequence.

\begin{corollary}\label{cor:sc}{\normalfont (cf.~\cite[Corollary 5.1.9]{CvRoSi-2})} If $\mu$ is not an eigenvalue of the signed graph $\Sigma'$
	 then there is a signed graph $\Sigma$ with $\Sigma'$ as a star
	complement for $\mu$ if and only if
	\begin{equation*}\langle{\bf b}_u,{\bf b}_u\rangle = \mu ~~~\text{and}~~~\langle{\bf b}_u,{\bf b}_v\rangle \in\{0, 1, -1\},\end{equation*}
	for all distinct $u,v\in S=V(\Sigma)\setminus V(\Sigma')$, where $\mathbf{b}_u$ and $\mathbf{b}_v$ determine the neighbourhoods of $u$ and $v$ in $\Sigma'$, respectively.
\end{corollary}

Clearly, the vectors $\mathbf{b}_u~(1\leq u\leq k)$ form the $t \times k$ 
submatrix $B$ in the Reconstruction Theorem.  Also, if $\langle{\bf b}_u,{\bf b}_v\rangle=0$ (resp.~$\langle{\bf b}_u,{\bf b}_v\rangle=-1$, $\langle{\bf b}_u,{\bf b}_v\rangle=1$), then $u\nsim v$ (resp.~$u\simp v$, $u\simm v$).

\section{A sharp upper bound}\label{sec:main}

	We retain the notation from the previous section; in particular, we assume that $\mu\notin\{0, 1, -1\}$ is an eigenvalue of a signed graph $\Sigma$. The dimension and codimension of $\mathcal{E}(\mu)$ are denoted by $k$ and $t$, respectively.
	
	We set $S = (B\, |\,C -\mu I)$ and denote the columns of $S$ by
	$\mathbf{s}_u$ $(1\leq u\leq n)$. In particular, we have $\mathbf{s}_u = \mathbf{b}_u$, for $1\leq u\leq k$. Then,
	$$\mu I- A_{\Sigma} = S^T(\mu I - C)^{-1}S.$$
	It follows that, for every pair of vertices of $\Sigma$, we have
	\begin{equation}\label{dot}
	\langle \mathbf{s}_u,\mathbf{s}_v\rangle=\left\{
	\begin{array}{rl}
	\mu, &u=v, \\
	0, & u\nsim v, \\
	-1, & u\simp v, \\
	1, & u\simm v.
	\end{array}
	\right.
	\end{equation}

	
	The following lemma is taken from \cite{BeRo}; the proof is basically unchanged.
	\begin{lemma} {\normalfont (cf.~\cite{BeRo})} \label{pq} {\rm Let
	${\bf w}=(w_1, w_2,\ldots, w_n)^T \in \mathcal{E}(\mu)^\perp$, and let 
	$\mathbf{q}=(w_{k+1}, w_{k+2}, \ldots, w_n)^T$. Then $$\langle 
	{\bf s}_u,\mathbf{q}\rangle=-w_u,~~ \text{for}~~1\leq u\leq n.$$}
	\end{lemma}
	
		To each vertex $u$ we associate a function $F_u\colon \mathbb{R}^t\longrightarrow \mathbb{R}$ defined by
		\begin{equation}\label{fun}
		F_u(\mathbf{x})=\langle \mathbf{s}_u,\mathbf{x}\rangle^3.
		\end{equation}
	
	
	


	\begin{lemma} \label{ind}
		The functions $F_1, F_2, \ldots, F_n$ defined by~\eqref{fun} are linearly independent.
	\end{lemma}
	
	\begin{proof} Using (\ref{dot}) and (\ref{fun}), we obtain
		$$F_u(\mathbf{s}_v)=\left\{
	\begin{array}{rl}
	\mu^3, & u=v, \\
	0, & u\nsim v, \\
	-1, & u\simp v, \\
	1, & u\simm v.
	\end{array}
	\right.
	$$
	
	Assume that there is a vector $\mathbf{a}=(\alpha_1,\alpha_2,\ldots,\alpha_n)^T\in\mathbb{R}^n$, such that $\sum_{u=1}^n \alpha_uF_u(\mathbf{x})=0$ holds for every $\mathbf{x}\in \mathbb{R}^t$. Then, for all $\mathbf{x}, \mathbf{y}\in \mathbb{R}^t$, we have $\sum_{u=1}^n \alpha_uF_u(\mathbf{x}+\mathbf{y})=0$. Hence,
	\begin{eqnarray*}
		\sum_{u=1}^n \alpha_u\langle \mathbf{s}_u,\mathbf{x}+\mathbf{y}\rangle^3 &=& \sum_{u=1}^n \alpha_u\langle \mathbf{s}_u,\mathbf{x}\rangle^3 +3\sum_{u=1}^n \alpha_u\langle \mathbf{s}_u,\mathbf{x}\rangle^2\langle \mathbf{s}_u,\mathbf{y}\rangle \\
		&&+  3\sum_{u=1}^n \alpha_u\langle \mathbf{s}_u,\mathbf{x}\rangle\langle \mathbf{s}_u,\mathbf{y}\rangle^2+\sum_{u=1}^n \alpha_u\langle \mathbf{s}_u,\mathbf{y}\rangle^3=0,
	\end{eqnarray*}
	i.e.,
	$$\sum_{u=1}^n \alpha_u\langle \mathbf{s}_u,\mathbf{x}\rangle^2\langle \mathbf{s}_u,\mathbf{y}\rangle + \sum_{u=1}^n \alpha_u\langle \mathbf{s}_u,\mathbf{x}\rangle\langle \mathbf{s}_u,\mathbf{y}\rangle^2=0.$$
	In the similar way, by computing $\sum_{u=1}^n \alpha_u\langle \mathbf{s}_u,\mathbf{x}-\mathbf{y}\rangle^3$, we obtain
	$$-\sum_{u=1}^n \alpha_u\langle \mathbf{s}_u,\mathbf{x}\rangle^2\langle \mathbf{s}_u,\mathbf{y}\rangle + \sum_{u=1}^n \alpha_u\langle \mathbf{s}_u,\mathbf{x}\rangle\langle \mathbf{s}_u,\mathbf{y}\rangle^2=0,$$
	which, together with the previous equality, gives
	\begin{equation}\label{3}
	\sum_{u=1}^n \alpha_u\langle \mathbf{s}_u,\mathbf{x}\rangle^2\langle \mathbf{s}_u,\mathbf{y}\rangle=0,~~\text{for all}~~ \mathbf{x}, \mathbf{y}\in \mathbb{R}^t.
	\end{equation}
	Now, let $\mathbf{w}=(w_1,w_2,\ldots,w_n)^T
	\in \mathcal{E}(\mu)^\perp$,  ${\bf q}=(w_{k+1},\ldots,w_n)^T$.  By Lemma~\ref{pq}, we have $\langle \mathbf{s}_u,\mathbf{q}\rangle=-w_u$, for  $1\leq u\leq n$. By setting $\mathbf{x}=\mathbf{q}$ and $\mathbf{y}=\mathbf{s}_i$ in (\ref{3}), we obtain \begin{equation*}\label{4}
	\sum_{u=1}^n \alpha_u{w_u}^2\langle \mathbf{s}_u,\mathbf{s}_i\rangle=0,
	\end{equation*} which implies that $$\mu \alpha_i{w_i}^2-\sum_{u\sim i}\alpha_u{w_u}^2\sigma(ui)=0.$$
	Then we have
	$$(\mu I-A_{\Sigma})\mathbf{a}^*=0,$$
	where $\mathbf{a}^*=(\alpha_1{w_1}^2, \alpha_2{w_2}^2, \ldots,\alpha_n{w_n}^2)^T$. On the other hand, it is also the case that $$(\mu^3 I-A_{\Sigma})\mathbf{a}=0.$$
	Since $\mu \neq \mu^3$, we have $\mathbf{a}^T\mathbf{a}^*=0$, whence $\sum_{u=1}^n {\alpha_u}^2{w_u}^2=0$. Therefore,  
$\alpha_u{w_u}=0$ holds for $1\leq u\leq n$. If $w_u=0$ for all $\mathbf{w}\in \mathcal{E}(\mu)^\perp$, then  
$\mathbf{e}_u$ is orthogonal to $\mathcal{E}(\mu)^{\perp}$, and so  
$A_{\Sigma}{\bf e}_u =\mu {\bf e}_u$, whence $\mu=0$, a contradiction. Thus, $\alpha_u=0$, for $1\leq u\leq n$, and so the functions $\langle \mathbf{s}_u,\mathbf{x}\rangle^3$ are linearly independent.  \end{proof}
	
	Using the foregoing lemma we arrive at the following result.
	
	\begin{theorem} \label{boun}
		Let $\Sigma$ be a signed graph with $n$ vertices, and let 
		$\mu$ be an eigenvalue of $\Sigma$ with multiplicity $k$. If $\mu\notin\{0, 1, -1\}$ then
		$$n\leq {t+2\choose 3},$$
		where $t=n-k$.
	\end{theorem}
	
	\begin{proof}
		The functions $F_1, F_2, \ldots, F_n$ are linearly independent and belong \linebreak to the space ${\cal H}_3$ of homogeneous cubic functions on $\mathbb{R}^t$. Since this space has dimension 
		${t+2\choose 3}$,  the result follows.
	\end{proof}

	In the particular case of graphs, there is a similar result stating that                  $n\leq {t+1\choose 2}$ holds whenever $t>2$ and
	$\mu\notin\{0, -1\}$ \cite{BeRo}. The problem of determining all the graphs which attain this bound is open; one example can be found in~\cite{CvRoSi-2}.  We provide some examples of signed graphs.
	
	\begin{example}The quadrangle with an odd number of negative edges has the eigenvalues $\sqrt{2}$ and $-\sqrt{2}$, both with multiplicity 2. Thus, the upper bound of Theorem~\ref{boun} is attained (for both eigenvalues and $(n, t)=(4, 2)$).
	\end{example}
		
		\begin{example}
			We recall that if $\mathbf{e}_1, \mathbf{e}_2, \ldots, \mathbf{e}_8$ are the vectors of the canonical basis of $\mathbb{R}^8$, then the root system $E_8$  consists of the 112 vectors (also known as the roots) of the form $\pm\mathbf{e}_i\pm\mathbf{e}_j$, for $1\leq i<j\leq 8$, and the additional 128 roots of the form $\frac{1}{2}\sum_{i=1}^{8}\pm \mathbf{e}_i$, where the number of positive summands is even. A system of positive roots is a subset $E_8^+$ of $E_8$, such that (1) for 
			$\pm\mathbf{x}\in E_8$, exactly one of 
			$\mathbf{x}, -\mathbf{x}$ belongs to $E_8^+$ and (2) for $\mathbf{x},\mathbf{y}\in E_8^+$, if $\mathbf{x}+\mathbf{y}\in E_8$ then  $\mathbf{x}+\mathbf{y}\in E_8^+$. (A set of positive roots of $E_8$ is given in~\cite{Gle}.) Now, by \cite{StaF}, if the columns of  $N$ are the 120 roots of $E_8^+$, then $N^TN-2I$ is the adjacency matrix of a signed graph with eigenvalues 28 (with multiplicity 8) and $-2$ (with multiplicity 112). Therefore, the upper bound of Theorem~\ref{boun} is attained (for $(\mu, n, t)=(-2, 120, 8)$).
			
			We remark that different sets of positive roots generate switching equivalent signed graphs and that the negation of any of them also attains our upper bound (for $\mu=2$).
			
			\end{example}
		
		\section{Two particular cases}\label{sec:pc}
	
Here we consider the particular cases in which the bound obtained in the previous section can be decreased.

First, if $\mu$ is a non-main eigenvalue (that is, if $\mathcal{E}(\mu)$ is orthogonal to the all-1 vector $\mathbf{j}_n\in\mathbb{R}^n$), then the bound of Theorem~\ref{boun} can be decreased by~1.
	
	\begin{theorem}\label{the:dec}
			Let $\Sigma$ be a signed graph with $n$ vertices, and let 
			$\mu$ be a non-main eigenvalue of $\Sigma$ with multiplicity $k<n-1$. If $\mu\notin\{0, 1, -1\}$ then
		 \begin{equation}\label{eq:ubou}
			n\leq {t+2\choose 3}-1,\end{equation}
			where $t=n-k$.
	\end{theorem}
	\begin{proof}
		Since $\mathbf{j}_n\in\mathcal{E}(\mu)^\perp$, Lemma \ref{pq} shows that 
		$\langle \mathbf{s}_u,\mathbf{j}\rangle=-1~(1\leq u\leq n)$, where $\mathbf{j}$ is the all-$1$ vector in $\mathbb{R}^t$. Set $F(\mathbf{x})=\langle \mathbf{j},\mathbf{x}\rangle^3$ and suppose by way of contradiction that $F,F_1,\ldots, F_n$ are linearly dependent.  By Lemma~\ref{ind}, $F_1, F_2,\ldots, F_n$ are linearly independent,   and so
		$F$  belongs to the span of these functions,  
		say  $F=\sum_{u=1}^{n} \beta_u F_u$, that is, 
		$$\langle \mathbf{j},\mathbf{x}\rangle^3=\sum_{u=1}^{n}\beta_u \langle \mathbf{s}_u, \mathbf{x}\rangle^3,~~\text{for all}~~ \mathbf{x}\in  \mathbb{R}^t.$$
	By the argument exploited in the proof of Lemma \ref{ind}, we obtain
	\begin{equation}\label{eq:bas}
	\langle \mathbf{j},\mathbf{x}\rangle^2\langle \mathbf{j},\mathbf{y}\rangle=\sum_{u=1}^{n}\beta_u\langle \mathbf{s}_u,\mathbf{x}\rangle^2\langle \mathbf{s}_u,\mathbf{y}\rangle,~~ \text{for all}~~ \mathbf{x}, \mathbf{y}\in  \mathbb{R}^t.\end{equation}
	
	By setting $\mathbf{x}\!=\!\mathbf{y}\!=\!\mathbf{s}_i,$ we obtain $-1=\mu^3 \beta_i-\sum_{u\sim i}\beta_u\sigma(ui)$, for $1\leq i\leq n$, that is
	\begin{equation}\label{11}
	\mathbf{j}_n=(-\mu^3 I+A_{\Sigma})\mathbf{b},
	\end{equation}
	where $\mathbf{b}=(\beta_1, \beta_2, \ldots,\beta_n)^T$. Similarly, by setting $\mathbf{x}=\mathbf{j}$, $\mathbf{y}=\mathbf{s}_i$, we obtain $-\langle \mathbf{j},\mathbf{j} \rangle^2=\sum_u \beta_u\langle \mathbf{s}_u,\mathbf{s}_i\rangle=\mu\beta_i-\sum_{u\sim i}\beta_u\sigma(ui)$, which implies
	\begin{equation}\label{12}
	-\langle \mathbf{j},\mathbf{j} \rangle^2\mathbf{j}_n=(\mu I-A_{\Sigma})\mathbf{b}.
	\end{equation}
	From (\ref{11}) and (\ref{12}), we obtain $$(1-\langle \mathbf{j},\mathbf{j} \rangle^2)\mathbf{j}_n=(\mu-\mu^3)\mathbf{b}.$$ Observe that $\mu-\mu^3\neq 0$, and so $\mathbf{b}=\beta \mathbf{j}_n$ with  
	$\beta=\frac{1-s^2}{\mu(1-\mu^2)}$, where $s=\langle\mathbf{j}, \mathbf{j}\rangle$. In other words, $\mathbf{b}$ is a multiple of $\mathbf{j_n}$, which, together with \eqref{eq:bas}, gives
	\begin{equation}\label{eq:bas1}
	\langle \mathbf{j},\mathbf{x}\rangle^2\langle \mathbf{j},\mathbf{y}\rangle=\beta\sum_{u=1}^{n}\langle \mathbf{s}_u,\mathbf{x}\rangle^2\langle \mathbf{s}_u,\mathbf{y}\rangle,~~ \text{for all}~~ \mathbf{x}, \mathbf{y}\in  \mathbb{R}^t.\end{equation}

	%
	
By setting $\mathbf{y}=\mathbf{j}$ in \eqref{eq:bas1}, we get
	\begin{equation}
	\label{eq:j}
	s\langle \mathbf{j},\mathbf{x}\rangle^2=-\beta \sum_{u=1}^n \langle \mathbf{s}_u, \mathbf{x}\rangle^2,~~\text{for all}~~ \mathbf{x}\in\mathbb{R}^t,\end{equation}
	\noindent which gives
		\begin{equation}
		\label{eq:jj}
		s^2\langle \mathbf{j},\mathbf{x}\rangle^2\langle \mathbf{j},\mathbf{y}\rangle^2=\beta^2 \sum_{u=1}^n \langle \mathbf{s}_u, \mathbf{x}\rangle^2\sum_{u=1}^n \langle \mathbf{s}_u, \mathbf{y}\rangle^2,~~\text{for all}~~ \mathbf{x}, \mathbf{y}\in\mathbb{R}^t.\end{equation}
		
		By \eqref{eq:j}, we also have
		$$
		s\langle \mathbf{j},\mathbf{x}+\mathbf{y}\rangle^2=-\beta \sum_{u=1}^n \langle \mathbf{s}_u, \mathbf{x}+\mathbf{y}\rangle^2,~~\text{for all}~~ \mathbf{x}, \mathbf{y}\in\mathbb{R}^t,$$
		i.e.,
		$$s(\langle \mathbf{j},\mathbf{x}\rangle^2+2\langle\mathbf{j},\mathbf{x}\rangle\langle\mathbf{j},\mathbf{y}\rangle+\langle\mathbf{j},\mathbf{y}\rangle^2)=-\beta \sum_{u=1}^{n}(\langle \mathbf{s_u},\mathbf{x}\rangle^2+2\langle\mathbf{s_u},\mathbf{x}\rangle\langle\mathbf{s_u},\mathbf{y}\rangle+\langle\mathbf{s_u},\mathbf{y}\rangle^2),$$
		giving
		$$s\langle\mathbf{j},\mathbf{x}\rangle\langle\mathbf{j},\mathbf{y}\rangle=-\beta \sum_{u=1}^{n}\langle\mathbf{s_u},\mathbf{x}\rangle\langle\mathbf{s_u},\mathbf{y}\rangle,$$
		and so
		\begin{equation}
		\label{eq:jjj}
		s^2\langle \mathbf{j},\mathbf{x}\rangle^2\langle \mathbf{j},\mathbf{y}\rangle^2=\beta^2 \Big(\sum_{u=1}^n \langle \mathbf{s}_u, \mathbf{x}\rangle \langle \mathbf{s}_u, \mathbf{y}\rangle\Big)^2.\end{equation}
		Now, $\beta\neq 0$ in view of \eqref{eq:bas1}, and so by \eqref{eq:jj} and \eqref{eq:jjj}, we see that
		$$\sum_{u=1}^n \langle \mathbf{s}_u, \mathbf{x}\rangle^2\sum_{u=1}^n \langle \mathbf{s}_u, \mathbf{y}\rangle^2=\Big(\sum_{u=1}^n \langle \mathbf{s}_u, \mathbf{x}\rangle \langle \mathbf{s}_u, \mathbf{y}\rangle\Big)^2$$
		holds for all $\mathbf{x}, \mathbf{y}\in \mathbb{R}^t$. In other words, a Cauchy-Schwarz bound is attained and we may use the following argument from \cite{BeRo}.  We have $\langle \mathbf{s}_u, \mathbf{x}\rangle=\gamma \langle \mathbf{s}_u, \mathbf{y}\rangle$, for some $\gamma=\gamma(\mathbf{x}, \mathbf{y})\in \mathbb{R}$. Then $\langle \mathbf{s}_u,\mathbf{x}-\gamma \mathbf{y}\rangle=0$, which gives $s_u^T(\mu I-C)^{-1}(\mathbf{x}-\gamma \mathbf{y})=0$, for $1\leq u\leq n$, and so $$(C-\mu I)(\mu I-C)^{-1}(\mathbf{x}-\gamma \mathbf{y})=\mathbf{0}.$$
		The last identity implies that $\mathbf{x}=\gamma  \mathbf{y}$ holds for all $\mathbf{x},\mathbf{y}\in \mathbb{R}^t$, which is possible only for $t=1$, contrary to the assumption. Consequently, $F$ does not belong to the span of $F_1, F_2, \ldots, F_n$, also contrary to assumption.  Hence, ${\cal H}_3$ contains $n+1$ linearly independent 
		functions $F,F_1, \ldots, F_n$,  and we have 
		$n+1\le {\rm dim}({\cal H}_3)={t+2\choose 3}.$
\end{proof}

We arrive immediately at the following consequence.

\begin{corollary}\label{cor:attains}Let a signed graph $\Sigma$ with $n$ vertices be switching equivalent to a net-regular signed graph $\Sigma'$ with net-degree $\varrho$. Also, let $\mu\notin\{0, 1, -1, \varrho\}$ be an eigenvalue of $\Sigma$ with multiplicity $k$, and set $t=n-k$. If $t\geq 2$, then
 \begin{equation*}
 n\leq {t+2\choose 3}-1.\end{equation*}
\end{corollary}

\begin{proof} Since $\Sigma$ and $\Sigma'$ share the same spectrum, we deduce that the multiplicity of $\mu$ in $\Sigma'$ is also $k$. Now, $\mu$ is  non-main in $\Sigma'$ since $\mathbf{j}_n\in\mathcal{E}(\varrho)$, and so the result follows by Theorem~\ref{the:dec}.
\end{proof}

If $\Sigma$ is a homogeneous signed graph, by \cite{BeRo} the upper bound \eqref{eq:ubou} reduces to $n\leq {t+1\choose 2}-1$, and the case of equality is fully resolved in the same reference (see also \cite[Theorem 5.3.3]{CvRoSi-1}). In what follows
we consider inhomogeneous signed graphs that attain the equality in \eqref{eq:ubou}. For this purpose, we need the following combinatorial definition taken from \cite{ZS}. We say that a signed graph $\Sigma$ is \textit{strongly regular}  whenever it is regular and satisfies the following
conditions:
\begin{enumerate}
	\item[(i)] $\Sigma$ is neither homogeneous complete nor totally disconnected,
	\item[(ii)] there exists $a\in \mathbb{Z}$ such that $\sum_{u\in N(i)\cap N(j)}\sigma(ui)\sigma(uj)=a$, for all $i\simp j$,
	\item[(iii)] there exists $b\in \mathbb{Z}$ such that $\sum_{u\in N(x)\cap N(y)}\sigma(ui)\sigma(uj)=b$ holds for all $i \simm j$ and
	\item[(iv)] there exists $c\in \mathbb{Z}$ such that $\rho\in \mathbb{Z}$ such that
	$\sum_{u\in N(i)\cap N(j)}\sigma(ui)\sigma(uj)=c$ holds for all $i \nsim j$.
\end{enumerate}

Resuming the previous notation, if equality is attained in \eqref{eq:ubou}, then the functions $F, F_1, \ldots, F_n$ form a basis for 
${\cal H}_3$. In particular, we have
$$\langle \mathbf{x},\mathbf{x}\rangle\langle \mathbf{j},\mathbf{x}\rangle=\sum_{u=1}^n\beta_u F_u(\mathbf{x})+\gamma F(\mathbf{x}),$$
for some $(\beta_1, \beta_2, \ldots, \beta_n)^T=\mathbf{b}\in \mathbb{R}^n$ and some $\gamma\in \mathbb{R}$. We first prove that $\mathbf{b}$ is a multiple of ${\bf j}_n$.

By considering $\mathbf{x}+\mathbf{y}$ and $\mathbf{x}-\mathbf{y}$ instead of $\mathbf{x}$, we obtain
\begin{equation}\label{eq:eq}
\langle \mathbf{x},\mathbf{x}\rangle\langle \mathbf{j},\mathbf{y}\rangle+2\langle \mathbf{x},\mathbf{y}\rangle\langle \mathbf{j},\mathbf{x}\rangle=3\sum_{u=1}^n \beta_u\langle \mathbf{s}_u,\mathbf{x}\rangle^2\langle \mathbf{s}_u,\mathbf{y}\rangle+3\gamma\langle \mathbf{j},\mathbf{x}\rangle^2\langle \mathbf{j},\mathbf{y}\rangle.\end{equation}

By setting $\mathbf{x=j}$, $\mathbf{y=s}_i$, we obtain 
\begin{equation*}\label{eq:nr}
3s^2\gamma-3s=3(\beta_i\mu-\sum_{u\sim i}\beta_u\sigma(ui)),\end{equation*} and so 
$$(s^2\gamma-s){\bf j}_n=(\mu I-A_{\Sigma})\mathbf{b},$$ 
where $\mathbf{b}=(\beta_1,\beta_2,\ldots,\beta_n)^T$. By setting $\mathbf{x=y=s}_i$ in \eqref{eq:eq}, we get
$$-\mu=\sum_{u=1}^n\beta_u\langle \mathbf{s}_u,\mathbf{s}_i\rangle^3-\gamma,$$ and so $$(\gamma-\mu) \mathbf{j}_n=(\mu^3 I-A_{\Sigma})\mathbf{b}.$$ 
It follows that $\mathbf{b}=\beta \mathbf{j}_n$, where $\beta=\frac{s^2-s-\gamma+\mu}{\mu(1-\mu^2)}$. 

Again, if we consider $\mathbf{x}+\mathbf{z}$ and $\mathbf{x}-\mathbf{z}$ instead of $\mathbf{x}$ in \eqref{eq:eq}, in a very similar way we arrive at
\begin{equation}\label{sxyz}
\langle \mathbf{x},\mathbf{y}\rangle\langle \mathbf{j},\mathbf{z}\rangle+\langle \mathbf{x},\mathbf{z}\rangle\langle \mathbf{j},\mathbf{y}\rangle+\langle \mathbf{y},\mathbf{z}\rangle\langle \mathbf{j},\mathbf{x}\rangle=3\beta\sum_{u=1}^n \langle \mathbf{s}_u,\mathbf{x}\rangle\langle \mathbf{s}_u,\mathbf{y}\rangle\langle \mathbf{s}_u,\mathbf{z}\rangle+3\gamma\langle \mathbf{j},\mathbf{x}\rangle\langle \mathbf{j},\mathbf{y}\rangle\langle \mathbf{j},\mathbf{z}\rangle.
\end{equation}

By setting  $\mathbf{x=s}_i$, $\mathbf{y=z=j}$, we obtain
\begin{equation}\label{Snr}
-3s=3\beta\mu-3\beta d^{\pm}_i-3\gamma s^2.\end{equation}

Similarly, for $\mathbf{x=y=s}_i$, $\mathbf{z=j}$ in \eqref{sxyz}, we obtain\begin{equation}\label{Sr}\mu s+2=-3\beta\mu^2-3\beta d_i+3\gamma s.\end{equation}

Finally, for $\mathbf{x=s}_i$, $\mathbf{y=s}_j$, $\mathbf{z=j}$, we find
$$\langle \mathbf{s}_i,\mathbf{s}_j\rangle\langle \mathbf{j},\mathbf{j}\rangle+\langle \mathbf{s}_i,\mathbf{j}\rangle\langle \mathbf{j},\mathbf{s}_j\rangle+\langle \mathbf{s}_j,\mathbf{j}\rangle\langle \mathbf{j},\mathbf{s}_i\rangle=3\beta\sum_{u=1}^n\langle \mathbf{s}_u,\mathbf{s}_i\rangle\langle \mathbf{s}_u,\mathbf{s}_j\rangle\langle \mathbf{s}_u,\mathbf{j}\rangle+3\gamma\langle \mathbf{j},\mathbf{s}_i\rangle\langle \mathbf{j},\mathbf{s}_j\rangle\langle \mathbf{j},\mathbf{j}\rangle.$$

For $i\nsim j$, we immediately obtain
\begin{equation}\label{Sinj} 3\beta \sum_{u\in N(i)\cap N(j)}\sigma(ui)\sigma(uj)=3\gamma s-2.\end{equation}

For $i\sim j$, we obtain

\begin{equation}\label{Sir}3\beta \sum_{u\in N(i)\cap N(j)}\sigma(ui)\sigma(uj)=6\sigma(ij)\beta\mu-2+\sigma(ij)s+3\gamma s.\end{equation}

Since $\Sigma$ is inhomogeneous, the last equality implies $\beta\neq 0$. Gathering the above results, we arrive at the following conclusion.

\begin{theorem}If equality in \eqref{eq:ubou} is attained for a signed graph $\Sigma$, then $\Sigma$ is a net-regular strongly regular signed graph with parameters\\
$a = (6\beta\mu-2+s+3\gamma s)/3\beta,~~~ 
b = (-6\beta\mu-2-s+3\gamma s)/3\beta, ~~~c = (3\gamma s -2)/3\beta$.
\end{theorem}

\begin{proof}Note that the upper bound is not attained for $t=2$, while for $t\geq 3$, $\Sigma$ must be inhomogeneous (as ${t+1\choose 2}<{t+2\choose 3}$). Then, by the above discussion, we have $\beta\neq 0$. The net-regularity follows from \eqref{Snr}. The strong regularity follows from \eqref{Sr}--\eqref{Sir}.
\end{proof}

It follows that if $\Sigma$ is a signed graph which attains the upper bound
in Corollary~\ref{cor:attains}, then its switching equivalence class contains a net-regular strongly regular signed graph. 


In the case of Theorem 4.3,  we have 
$c=\frac12 (a+b)$.
From \cite{KS} we know that, under this condition, either $\Sigma$ has exactly two eigenvalues or it has exactly three  eigenvalues,  one of them being the net-degree with multiplicity 1.  Conversely,  if a net-regular signed graph of order $n$ has spectrum $\varrho,\lambda^{(f)}, \mu^{(g)}~(f\le g)$,
then a star complement for $\mu$ has order $f+1$,  and so
$ n\leq {f+3\choose 3}-1$, by Theorem \ref{the:dec}:  this inequality can be seen as an analogue of Seidel's `absolute bound' for strongly regular graphs.

We conclude with another particular case in which the upper bound of Theorem~\ref{boun} can be significantly improved.

	\begin{theorem}\label{betb}
		Let $\Sigma$ be a signed graph with $n$ vertices, let 
		$\mu\notin\{0, 1, -1\}$ be an eigenvalue of $\Sigma$  with multiplicity $k$,  and let $t=n-k$. If $-\mu^2$ is not an eigenvalue of the underlying graph $G$, then
		$$n\leq {t+1\choose 2}.$$
	\end{theorem}
	\begin{proof}
		 We define functions $P_u\colon \mathbb{R}^t\longrightarrow \mathbb{R}$ by
		 $$P_u(\mathbf{x})=\langle \mathbf{s}_u,\mathbf{x}\rangle^2,$$
		 so that $$P_u(\mathbf{s}_v)=\left\{
	\begin{array}{rl}
	\mu^2, & u=v, \\
	0, & u\nsim v, \\
	1, & u\sim v.
	\end{array}
	\right.
	$$
	This implies that $$(P_u(\mathbf{s}_v))_{u,v\in V(\Sigma)}=\mu^2I+A_{G}.$$ Since $-\mu^2$ is not an eigenvalue of $G$, the functions $P_1, P_2, \ldots, P_n$ are linearly independent. The assertion follows since the dimension of the space ${\cal H}_2$ of homogeneous quadratic functions on $\mathbb{R}^t$ is ${t+1\choose 2}$.
	\end{proof}
	
	Of course, the previous theorem improves the bound of Theorem~\ref{boun} whenever $t\geq 3$.

\section*{Acknowledgements}

Research of the third author is partially supported by Serbian Ministry of Education, Science and Technological Development, Projects 174012 and 174033.

\end{document}